\documentclass[12pt]{article}

\RequirePackage{amsmath}
\RequirePackage{amssymb}

\def\square{$\vcenter{\hrule\hbox{\vrule height 2truemm \kern 2truemm
\vrule }\hrule}$}
\newenvironment{proof}
{\noindent {{\bf Proof}:}}{\hspace*{\fill}\square\vskip 8pt}

\newcommand{\C}{{\bf C}}
\newcommand{\R}{{\bf R}}
\newcommand{\Z}{{\bf Z}}

\input xy
\xyoption{all}

\def\frac#1#2{{\textstyle{{#1} \overwithdelims.. {#2}}}}

\newtheorem{prop}{Proposition}[section]
\newtheorem{thm}[prop]{Theorem}
\newtheorem{lem}[prop]{Lemma}

\newtheorem{defi}[prop]{Definition}
\newtheorem{rem}[prop]{Remark}
\newtheorem{rems}[prop]{Remarks}

\begin{document}
\title{\bf On the curvature of the Real Amoeba \footnote{ AMS
Classification 14H20, 14h50, 14P59, . Keywords : Real Plane Algebraic
Curves, Curvature, }}
\author{Mikael Passare\\
Department of Mathematics\\
Stockholm University\\
SE-106 91 Stockholm\\
passare@math.su.se
   \and 
           Jean-Jacques Risler\\
           IMJ, UPMC\\
4, place Jussieu, case 247\\
75252 Paris Cedex 05\\
risler@math.jussieu.fr
} 

\date{\today}
\maketitle

\begin{abstract}
For a real smooth algebraic curve $A \subset (\C^*)^2$, the amoeba ${\cal A} \subset \R^2$ is the image of $A$ under the map Log : $(x,y) \mapsto (\log |x|, \log | y |)$. We describe an universal bound for the total curvature of  the real amoeba ${\cal A}_{\R A}$ and we prove that this bound is reached if and only if the curve $A$ is a simple Harnack curve in the sense of Mikhalkin.
\end{abstract}

 
\section{Introduction}

Let $A \subset (\C^*)^2$ be a smooth real algebraic curve with a non-degenerated Newton polygon $\Delta$ defined by an equation $f = 0$, where $f : \C^2 \rightarrow \C$ is a reduced polynomial with real coefficients.\\
$\R A \subset (\R^*)^2$ stands for the real part of $A$.\\
Let $Log$ be the map : $(\C^*)^2 \rightarrow \R^2$, $(z_1, z_2) \mapsto (\log |z_1|, \log |z_2|)$, $L = Log_{| A}$ the restriction of Log to the curve $A$, ${\cal A}_A = L(A), \; {\cal A}_{\R A}= L(\R A), \; F \subset A$  the critical set of $L $.\\
The set ${\cal A}_A $ is the {\it Amoeba} of $A$, and ${\cal A}_{\R A}$ its {\it Real Amoeba} (see \cite{GKZ}). Notice that ${\cal A}_{\R A}$ is a real curve since for each quadrant $Q \subset (\R^*)^2$, $L_{|  Q}$ is a diffeomorphism.\\
G. Mikhalkin (\cite{M}) defined  ``Simple Harnack Curves'' ( see below), and used the notion of Amoeba to prove the uniqueness of the topological type of the pair $(\bar A, T_{\Delta})$ when $A$ is a simple Harnack curve and $\bar A$ its closure in the toric surface $T_{\Delta}$.\\
In the paper \cite{P-R},  it is proved that ${\rm Area}({\cal A}) \leq {\pi}^2 {\rm Area}(\Delta)$ and in \cite{M-R} the auhors prove $A$ is a simple Harnack curve if and only if there is equality above, i.e., if and only if the Area of ${\cal A}$ is maximal.\\
In this paper, we prove a similar result, but for the total curvature of the real Amoeba (instead of the volume of the Amoeba) ; notice that for a hypersurface (in higher dimension), the total curvature of the Real Amoeba is always finite, but not the volume of the Amoeba which is not bounded in general. We then hope to generalize the results of this paper in higher dimension (at least for surfaces).

\section{Amoebas and Curvature}

Let $T_{\Delta}$ be the toric surface associated to $\Delta$, $\tilde l_i \; (1 \leq i \leq n)$ the sides of $\Delta$, $d_i$ the integer length of $\tilde l_i$, $\bar A$  (resp. $ \R \bar A)$ the closure of the image of $A$ in $T_{\Delta}$ 
(resp. in $\R T_{\Delta}$). \\
We denote by $ l_i$ the component of the divisor $T_{\Delta}^{\infty} : = T_{\Delta} \setminus ((\C^*)^2)$ corresponding to the side $\tilde l_i \in \partial \Delta$. We assume that $\bar A$ is smooth and transverse to $T_{\Delta}^{\infty} $.\\

Set $\gamma$ for the logarithmic Gauss map : $A \rightarrow \C P^1$ (see \cite{M}); the map $\gamma$ is defined by $\gamma (z_1,z_2) = [ z_1 \partial f/\partial z_1, z_2 \partial f/ \partial z_2 ].$\\
The real Logarithmic Gauss map $\gamma_{\R} : \R A \rightarrow  \R P^1$ is the restriction of $\gamma$ to $\R A$.\\
We have $F = \gamma^{-1}(\R P^1)$ (\cite{M}, Lemma 3), and the following commutative diagram :

\[
 \xymatrix{ \R A \ar[d]_{L} \ar[dr]^{\gamma_{\R}}\\
 {\cal A}_{\R A}  \ar[r]_g & \R P^1 } \]

\noindent where $g$ is the usual Gauss map, defined on the smooth part of ${\cal A}_{\R A}$.\\
The logarithmic Gauss map $\gamma : A \rightarrow \C P^1$ extends to a map $\bar \gamma : \bar A \rightarrow \C P^1$ of degree $2 vol(\Delta)$ (\cite{M}), therefore
the map $\gamma_{\R}$  has finite fibers of cardinal $\leq 2 vol(\Delta)$ which implies that the fibers of $g$ are also of cardinal $\leq 2 vol(\Delta)$. \\

If $k$ denotes the curvature function on the curve ${\cal A}_{\R A}$ (for any orientation), we have 
\begin{equation} \label{1}
  \int_{{\cal A}_{\R A}} \vert k \vert  \leq  2 \pi vol (\Delta) 
\end{equation}
(see \cite{R} or \cite{L}), since vol(${\R} P^1) = \pi$ and because  $ \int_{{\cal A}_{\R A}} \vert k \vert$ is the volume of ${\rm im}(g) \subset \R P^1$ (counted with multiplicities, the multiplicity of  $x \in \R P^1$ being the cardinal of the fiber $g^{-1}(x)$ which is $\leq 2 vol(\Delta)$).        

\begin{defi} \label{Mcurv}

We say that the real amoeba ${\cal A}_{\R A}$ has maximal curvature if there is equality in (\ref{1}).

\end{defi}

\begin{lem} \label{maxcurv}
Let $A \subset (\C^*)^2$ be a real smooth algebraic curve. Then the following conditions are equivalent:
\begin{enumerate}
\item ${\cal A}_{\R A}$ has maximal curvature
\item The logarithmic Gauss map $\gamma : A \rightarrow \C P^1$ is totally real (i.e., $\gamma^{-1}(x) \subset \R A$ for $x \in \R P^1$).
\end{enumerate}

\end{lem}
\begin{proof}
$1. \Rightarrow 2.$  If $\gamma$ is not  totally real, there would exist an open set $U \subset \R P^1$ such that ${\rm card}(\gamma^{-1}_{\R}(x)) < 2 vol (\Delta)$  $\forall x \in U$; then we would have also ${\rm card}(g^{-1}(x)) < 2 vol(\Delta) \; \forall x \in U$, which implies that ${\rm  vol}({\rm im}(g)) < 2 \pi vol(\Delta)$.\\
$2. \Rightarrow 1.$ Since the map $\gamma$ is totally real, we have that the fibers of $g$ are generically of cardinal $2 vol(\Delta)$, which implies that  $ \int_{{\cal A}_{\R A}} \vert k \vert ={\rm  vol}({\rm im}(g)) = 2 \pi {\rm vol}(\Delta)$, i.e.,  the real amoeba ${\cal A}_{\R A}$ has maximal curvature.
\end{proof}
\begin{rems} \label{maxtotcurv}
\end{rems}

 Notice that if ${\cal A}_{\R A}$ has maximal curvature, then :
\begin{itemize}
\item
For $x \in \R P^1,$ $\gamma^{- 1}(x) \subset \R A$, which implies $F \subset \R A$ since $F = \gamma^{- 1}(\R P^1)$ (\cite{M}) and therefore $F = \R A$. 


\item  The only possibilities for ${\cal A}_{\R A}$ to have inflection points are at ``pinching points'' (cf. \cite{M}), which are also the only possible singular points of ${\cal A}_{\R A}$ (we will see that in fact if  ${\cal A}_{\R A}$ has maximal curvature, then ${\cal A}_{\R A}$ is smooth).

\end{itemize}


\section{Simple Harnack curves}

\begin{defi}  1) A real Algebraic smooth curve $A \subset (\C^*)^2$ with Newton polygon $\Delta$ is said {\rm maximal} (or a {\rm M-curve}) if the number of connected components of $\R \bar A$ is $g + 1$, where $g$ is the genus of $\bar  A$ (or the number of integer points in  the interior of $\Delta$). \\
2) $A$ is said "In weak maximal position" if $\bar A$
is smooth and transverse to $T_{\Delta}^{\infty} $ and if  for $1 \leq  i \leq n$, $\bar A \cap \R l_i = \bar A \cap  l_i$  (which means that $\bar A \cap T_{\Delta}^{\infty} $ is totally real). 
\end{defi}.

Let us look at the following properties for a M-curve $A$ in weak maximal position.

\begin{enumerate}

\item There exists a connected component $C \subset  {\R \bar A}$ such that for $1 \leq i \leq n$, $C \cap \R  l_i = \bar A \cap  l_i $ (i.e. $C \cap \R  l_i$ is made of $d_i$ points).

\item Condition 1. is verified and for $1 \leq i \leq n$, $C \cap \R  l_i$ is contained in an arc $C_i$ of $C$, such that $C_i  \cap C_j = \emptyset$ for $i \not = j$.


\item Condition 2. is verified and there is an orientation of $\partial \Delta$ such that $C$ is cyclicly in maximal position (see \cite{M}).

\end{enumerate}

Notice that Condition 3. is the definition of a "Simple Harnack curve" in the sense of \cite{M}.\\
\begin{rems}
\end{rems}
 \begin{enumerate}
\item  It is proved in \cite{M} that a simple Harnack curve verifies also the following property:\\
 For $1 \leq i \leq n$ the order of the points $C_i \cap \R l_i$ is the same on $C_i$ and $\R  l_i$ (for some orientations).

\item An alternative to conditions 2. and  3. is given by the following lemma:
\begin{lem} \label{cond3}
With the above notations, assume that $A$ is a $M$-curve in weak maximal position and that condition {\rm 1.} is  fulfilled. Then  the following conditions are equivalent:
\begin{enumerate}
\item The curve $A$ is a simple Harnack curve (i.e., conditions {\rm 2.}  and {\rm 3.} are realised)
\item The amoeba ${\cal A}_{\R A}$ is smooth.
\end{enumerate}
\end{lem}
\begin{proof} \\
(a) $\Rightarrow$ (b) is proved in \cite{M} (Corollary 9)\\
(b) $\Rightarrow$ (a) 
We have to prove that conditions 2. and 3. above are fulfilled if the amoeba  ${\cal A}_{\R A}$ is smooth. \\
Assume by absurd that 2. is not fulfilled for a component $l_i$ of $T_{\Delta}^{\infty}$. Using the moment map, we can see  (topologically) the amoeba ${\cal A}_{\R A}$ as a subset of 
$\overset{\circ}{\Delta} $. If $C \subset \R A$, set $\tilde C$ for its image in $\overset{\circ}{\Delta} $.\\ If 2. is not fulfilled
, there would exist an arc $\tilde C_1$ of $\tilde  C$ joining the side $\tilde  l_i$ of $\partial \Delta$ to another one $\tilde  l_j$ and such that if $\overline{\tilde C_1} \cap \tilde  l_i = \{ P \}$, there are points of $\tilde  C \cap \tilde  l_i$ on both sides of $P$.\\
Therefore $\Delta \setminus \tilde  C_1$ has two connected components, and there must exist an arc $\tilde  C' \subset \tilde  C$ joining theses two components hence intersecting $\tilde  C_1$, which implies that the amoeba ${\cal A}_{\R A}$ is not smooth.\\
The proof of condition 3. is similar.
\end{proof}

\end{enumerate}

\begin{prop} If $A$ is a Simple Harnack curve, then ${\cal A}_{\R A}$ has maximal curvature (Definition \ref{Mcurv}).
\end{prop}
\begin{proof}
The logarithmic Gauss map  (restricted to the real part) $ \gamma_{\R}  : \; \R \bar A \rightarrow \R P^1$ is totally real (i.e., the cardinal of each fiber is $2 vol(\Delta)$)  (see Lemma \ref{maxcurv}). Therefore it is also true for the Gauss map $g $ :   ${\cal A}_{\R A} \rightarrow \R P^1$ ($\gamma_{\R} = g \circ L$) since the hypothesis implies that $L \vert \R A$ is a bijection onto ${\cal A}_{\R A}$; then the total curvature of  ${\cal A}_{\R A}$ is  vol(Im($g))= 2 \pi vol(\Delta)$.
\end{proof}

\begin{thm} \label{thm1} Let $A \subset (\C^*)^2$ be a smooth real algebraic curve. Assume that ${\cal A}_{\R A}$ has maximal curvature. Then $A$ is a Simple Harnack curve.
\end{thm}
\begin{proof}
Let us first prove the theorem under the supplementary assumption that  the real amoeba ${\cal A}_{\R A}$ is smooth.\\
Let $p$ be the number of compact connected  components of ${\cal A}_{\R A}$. Let  $t$ be the cardinal of $\R \bar A \cap T_{\Delta}^{\infty}$, $s = \sum d_i = {\rm card}(\partial \Delta \cap \Z^2)$. \\
Set $g = {\rm card}( \overset{\circ}{\Delta} \cap \Z^2)$; $g$ is the genus of the curve $\bar A$.\\

a)  By  Remarks \ref{maxtotcurv} and the smoothness assumption, the real Amoeba has no inflection point.  Let us prove that since ${\cal A}_{\R A}$ is smooth and has no inflection  points, one has 
\begin{equation} \label{totcurv}
 \int_{{\cal A}_{\R A}} \vert k \vert    \leq 2\pi p + \pi t 
\end{equation}

There are exactly $t$  arcs  of ${\cal A}_{\R A}$ which are not compact, each one of total curvature $\leq \pi$ (since it has no inflection point), which proves (\ref{totcurv}) since each compact component have total curvature equal to $2 \pi$.\\

b) Recall  Pick's formula :
\begin{equation} \label{Pick}
 vol(\Delta) = g + s/2 - 1 
\end{equation}

The hypothesis of maximal curvature and Pick's formula implie :

\[ \int_{{\cal A}_{\R A}} \vert k \vert = 2 \pi vol(\Delta)) = 2 \pi g + \pi s - 2\pi \]
wich gives, applying (\ref{totcurv}):
\begin{equation} \label{ineg}
2\pi g + \pi s - 2\pi \leq 2\pi p + \pi t.
\end{equation}

  Let us prove that  $A$ is a M-curve in weak maximal position. We have necesseraly $t > 0$ because otherwise $s \geq 6$, because the integer length of each side of $\Delta$ would then be even, which is not possible by (\ref{ineg}) since $p \leq g+1$.\\
We have therefore $p \leq g$.\\
\par  (a)   Let us first prove that $t = s$ ($t \leq s$ by definition). If we assume by absurd that $t \leq s-1$, we get from (\ref{ineg}) that $2 \pi g \leq 2\pi p + \pi$, or $2g \leq 2p + 1$ which implies $g = p$.\\
As above, using the moment map, we can see topologically the Amoeba ${\cal A}_{\R A}$ as a subset of $\overset{\circ}{\Delta} $. \\
Let us denote by $c_{ij} \; (i \not = j)$ an arc of ${\cal A}_{\R A}$ (if it exists) joining $\tilde l_i$ to $\tilde l_j$; if $\alpha_{ij}$ denotes the angle between $\tilde l_i$ and $\tilde l_j$, then the total curvature of $c_{ij}$ is $\alpha_{ij}$.\\
Since $g = p$, there exists only one connected component $C \subset \R \bar A$ such that $C \cap T_{\Delta}^{\infty} \not = \emptyset$. Therefore if $C$ intersect $k$ components $ l_r$ ($k \leq n$) of the divisor $T_{\Delta}^{\infty}$, there exist  $k$ arcs $c_{r r'}$
of  ${\cal A}_{\R A}$ such that each side $\tilde l_r$ intersect two of them. Since the sum of the angles between two consecutive sides of $\Delta$ is $(n-2) \pi$, we have $\sum \alpha_{ij} \leq (n-2) \pi$. Now there is at most $t-n$ other non compact arcs in ${\cal A}_{\R A}$, each one of total curvature $\leq \pi$ (since they have no inflection point), 
which gives 
\[ 2\pi g + \pi s - 2\pi = \int_{{\cal A}_{\R A}} \vert k \vert \leq 2\pi g+ \pi(t-n) + \pi(n-2) = 2\pi g + \pi t - 2 \pi \]
or $s = t$, contrary to the hypothesis.
\par (b) The relation (\ref{ineg}) implies now that $2\pi g + \pi s - 2 \pi \leq 2 \pi p + \pi s$ (since $t = s$) which implies $g-1 \leq p \leq g$. But since $p \geq g-1$, there is at least one arc $c$ of $\R \bar A$ which intersects two different components $l_i$ and $l_j$ of $T_{\Delta}^{\infty}$; therefore the total curvature of the corresponding arc of ${\cal A}_{\R A}$ is
$< \pi$. We have then $2 \pi g + \pi s - 2 \pi < 2 \pi p + \pi s$, which implies that necesseraly $p = g$.\\
We have then that $A$ is a M-curve in weak maximal position satisfying conditon 1.\\
Now $A$ is a simple Harnack curve by Lemma \ref{cond3}, since we have assumed ${\cal A}_{\R A}$ smooth.
\end{proof}


\section{On the smoothness of the real Amoeba}

To achieve the proof of Theorem \ref{thm1}, We must prove that under the hypothesis of maximal curvature, the real amoeba ${\cal A}_{\R A}$ is smooth.\\ Since $A$ is smooth and that ${\cal A}_{\R A}$ does not have smooth inflection point (Remark \ref{maxtotcurv}), the only {\it a priori} possible inflection points (or singular points) for ${\cal A}_{\R A}$ are  pinching points (see \cite{M}).  \\At a pinching point $P$ there are two smooth branches of ${\cal A}_{\R A}$ crossing at $P$, each one with an inflection point at $P$ (see \cite{M}). Let $\alpha$ ($0 \leq \alpha \leq \pi/2$) be the angle  of these two branches at $P$.
\begin{lem} \label{mxcurv} Assume that ${\cal A}_{\R A}$ has a pinching point $P$. Then the real amoeba ${\cal A}_{\R A}$ does not has maximal curvature (i.e., there exits $\epsilon  >  0$ such that 
\[  \int_{{\cal A}_{\R A}} \vert k \vert \leq   2 \pi g + \pi (s - 2) - \epsilon.) \]
\end{lem} 
\begin{proof} let $B_P \subset \R^2$ be a ball centered at $P$ small enough such that it contains no other singular point of ${\cal A}_{\R A}$ than the point $P$.  Since $P$ is a pinching point and the curve $A$ is smooth, we have that $L^{-1}(P) = \{ P_1, P_2 \} $ where $P_i \in \R A$ are two points in different quadrants of $(\R^*)^2$. Then $L^{-1}({\cal A}_{\R A} \cap B_P)$ has two connected components $C_1$ and $C_2$ (in different quadrants of $(\R^*)^2$) such that $P_1 \in C_1$ and $P_2 \in C_2$. Let $\gamma_1$ be the logarithmic Gauss map restricted to $C_1$ : since $C_1$ has by hypothesis  a logarithmic inflection point at $P_1$, the map $\gamma_1$ has a local extremum at $P_1$ (in a chart of  $\R P^1$). If $\alpha_1 = \gamma_1(P_1)$, ($\alpha_1 \in ] - \pi/2, \pi/2 [$,  and if $\alpha_1$ is for instance a local maximum at $P_1$ for $\gamma_1$, then there exists $\varepsilon > 0$ such that for $\theta \in ] \alpha_1, \alpha_1 + \varepsilon [, \; \gamma_1^{-1}(\theta)$ has two non real (conjugate) points in $L^{-1}(B_P)$. \\
Then the logarithmic Gauss map $\gamma :  A \rightarrow \C P^1$ is not totally real, and the amoeba ${\cal A}_{\R A}$ does not have maximal curvature (Lemma \ref{maxcurv}).
\end{proof}
\begin{rem}
\end{rem}
If the real amoeba ${\cal A}_{\R A}$ has not maximal curvature, it may have a pinching point $P$.\\
However, if  $\alpha$ is the angle between the two branches of ${\cal A}_{\R A}$ crossing at $P$, it is possible to prove that the angle $\alpha$ is necesserally $> 0$.

\end{document}